\documentclass[11pt]{amsart}
\usepackage{amsmath}
\usepackage{amssymb}
\usepackage{amsthm}
\usepackage{latexsym}

\newtheorem{Theorem}{Theorem}[section]
\newtheorem{Lemma}[Theorem]{Lemma}
\newtheorem{Corollary}[Theorem]{Corollary}
\newtheorem{Proposition}[Theorem]{Proposition}
\newtheorem{Remark}[Theorem]{Remark}

\newtheorem{Definition}[Theorem]{Definition}

\def\NZQ{\mathbb}               % the font for N,Z,Q,R,C
\def\NN{{\NZQ N}}

\def\ZZ{{\NZQ Z}}

\def\opn#1#2{\def#1{\operatorname{#2}}}
\opn\depth{depth}
\opn\dim{dim}
\opn\link{link}
\opn\type{type}
\opn\reg{reg}
\opn\height{height}
\opn\pd{pd}
\opn\Tor{Tor}
\opn\Ass{Ass}
\opn\multideg{multideg}

\def\FF{{\mathcal F}}

\def\PP{{\mathcal P}}

\numberwithin{equation}{section}

\begin{document}

\title{Sequentially Cohen-Macaulay\\ mixed product ideals}
\author{Giancarlo RINALDO }
\address{Giancarlo Rinaldo, Dipartimento di Matematica, Universit\`a di Messina, 
Viale Ferdinando Stagno d'Alcontres, 31, 98166 Messina. Italy}
\email{giancarlo.rinaldo@tiscali.it}



\begin{abstract}
We classify the ideals of mixed products that are sequentially Cohen-Macaulay.
\end{abstract}

\maketitle

\section{Introduction}
\label{1}

The class of ideals of mixed products is a special class of square-free monomial ideals. They were first introduced by G.~Restuccia and R.~Villarreal (see \cite{RV} and \cite{Vi}), who studied the normality of such ideals.

In \cite{IR}, C.~Ionescu and G.~Rinaldo studied the Castelnuovo-Mumford regularity, the depth and dimension of mixed product ideals and characterize when they are Cohen-Macaulay. In \cite{Ri} the author calculated the Betti numbers of their finite free resolutions. In \cite{HN}, L. T. Hoa and N. D. Tam studied these ideals in a broader situation.

Let $S=K[\mathbf{x},\mathbf{y}]$ be a polynomial ring over a field $K$ in two disjoint
sets of variables $\mathbf{x}=\{x_1,\dots,x_n\}$, $\mathbf{y}=\{y_1,\dots,y_m\}$. The \textit{ideals of mixed products} are the proper ideals 
\begin{equation}\label{def:mixprod}
\sum_{i=1}^s  I_{q_i} J_{r_i}\ \ q_i,r_i\in \ZZ_{\geq 0} 
\end{equation}
where $I_{q_i}$ (resp. $J_{r_i}$) is the ideal of $S$ generated by all the
square-free monomials of degree $q_i$ (resp. $r_i$) in the variables $\mathbf{x}$
(resp. $\mathbf{y}$). We set  $I_0=J_0=S$ and $I_{q_i}=(0)$ (resp. $J_{r_i}=(0)$) if $q_i>n$ (resp. $r_i>m$). 
In the articles mentioned only two summands of \ref{def:mixprod} are allowed. In this article we classify the ideals of mixed product that are sequentially  Cohen-Macaluay and Cohen-Macaulay for any $s\in \NN$. Recently, a number of authors have been interested in classifying sequentially Cohen-Macaulay rings related to combinatorial structures (for example see \cite{Fa}, \cite{HH}, \cite{VV}). This paper is inserted in this area and the tools used are essentially Stanley-Reisner rings and Alexander dual.

In section \ref{sec:pre} we recall some preliminaries about simplicial complexes and questions related to commutative algebra. In section \ref{sec:Alex} we study the primary decomposition of mixed product ideals, we introduce the vectors $\bar{q}$, $\bar{r}$ that are uniquely determined by the values $q_i$, $r_i$ for $i=1,\ldots,s$, $m$ and $n$, and we classify the Cohen-Macaulay mixed product ideals in terms of the vectors $\bar{q}$ and $\bar{r}$. The vectors $\bar{q}$ and $\bar{r}$ are used also to  classify the sequentially Cohen-Macaulay mixed product ideals in the last section.

\section{Preliminaries}\label{sec:pre}
In this section we recall some concepts on simplicial complexes that we will use in the article (see \cite{BH}, \cite{MS}, \cite{St}).

Set $V = \{x_1, \ldots, x_n\}$. A \textit{simplicial complex} $\Delta$ on the vertex set $V$ is a collection of subsets of $V$
such that (i) $\{x_i\} \in \Delta$  for all $x_i \in V$ and (ii) $F \in \Delta$ and $G\subseteq F$ imply $G \in \Delta$. An
element $F \in \Delta$ is called a \textit{face} of $\Delta$. For $F \subset V$ we define the \textit{dimension} of $F$ by $\dim F
=  |F| -1$, where $|F|$ is the cardinality of the set $F$. A maximal face of $\Delta$  with respect to inclusion is
called a \textit{facet} of $\Delta$.
%The number of facets of $\Delta$ is called the \textit{arithmetic degree} of $K[\Delta]$
%and denoted by $\arithdeg \,K[\Delta]$.
If all facets of $\Delta$ have the same dimension, then $\Delta$ is called \textit{pure}.

A simplicial complex $\Delta$ is called \textit{shellable} if the facets of $\Delta$ can be given a linear order $F_1, \ldots, F_t$ such that for all $1 \leq i < j \leq t$, there exist some $v \in F_j \setminus F_i$ and some $k \in \{1, \dots, j-1\}$ with $F_j\setminus F_k = \{v\}$.

Moreover, a pure simplicial complex $\Delta$ is \textit{strongly connected} if for every two facets $F$ and $G$ of $\Delta$  there is a sequence of facets $F = F_0, F_1, \ldots, F_t = G$ such that $\dim ( F_i \cap F_{i+1}) = \dim \Delta -1$ for each $i =0, \ldots, t-1$.

The \textit{Stanley--Reisner ideal} of $\Delta$, denoted by $I_{\Delta}$, is the squarefree monomial ideal of $S=K[x_1,\ldots,x_n]$ generated by 
\[
 \{x_{i_1} x_{i_2} \cdots x_{i_p} \,:\, 1 \le i_1 < \cdots < i_p \le n,\; 
\{x_{i_1},\ldots,x_{i_p}\} \notin \Delta \},  
\]
and $K[\Delta]= K[x_1,\ldots,x_n]/I_{\Delta}$ is called the \textit{Stanley--Reisner ring} of $\Delta$. It is known that
\begin{equation}\label{eq:stanleyideal}
  I_\Delta=\bigcap_{F\in \FF(\Delta)} P_F
\end{equation}
with $P_F=(\{x_1\,\ldots,x_n\}\setminus F)$.

Let $I=(\mathbf{x}^{\alpha_1},\ldots, \mathbf{x}^{\alpha_q})\subset K[\mathbf{x}]= K[x_1,\ldots,x_n]$ be a square-free  monomial ideal, with $\alpha_i=(\alpha_{i_1},\ldots,\alpha_{i_n})\in \{0,1\}^n$. The \textit{Alexander dual} of $I$ is the ideal  
\begin{equation}\label{eq:alexdual}
 I^*=\mathop\bigcap\limits_{i=1}^q \mathfrak{m}_{\alpha_i},  
\end{equation}
  where $\mathfrak{m}_{\alpha_i}=(x_j:\alpha_{i_j}=1)$. It is known that $(I^*)^*=I$. We also have that if  $I$, $J$ are squarefree monomial ideals of $S=K[x_1,\ldots,x_n]$ then
\begin{equation}\label{eq:sumdual}
 (I+J)^*=I^*\cap J^*.
\end{equation}

\section{Cohen-Macaulay mixed product ideals}\label{sec:Alex}
Let $S=K[x_1,\ldots,x_n,y_1,\ldots,y_m]$ be a polynomial ring over a field $K$ and let 
\begin{equation}\label{def:mixprodcopy}
\sum_{i=1}^s  I_{q_i} J_{r_i},\ \ q_i,r_i\in \ZZ_{\geq 0},\ s\in \NN, 
\end{equation}
be an ideal of mixed product as in \ref{def:mixprod}. In this section we study the primary decomposition of the ideal \ref{def:mixprodcopy} and give a criterion for its Cohen-Macaulayness. Under the assumption that no summands in \ref{def:mixprodcopy} is a subset of another summand, we  set
\begin{equation}\label{order:q}
 0 \leq q_1< q_2<\ldots<q_s\leq n.
\end{equation}
Under this assumption and because of the ordering \ref{order:q} we have
\begin{equation}\label{order:r}
 0\leq r_s< r_{s-1} <\ldots<r_1\leq m.
\end{equation}
Throughout this paper we always assume \ref{order:q} and \ref{order:r}. 
%We recall the following (see \cite{MS}, Definition 1.35):
% \begin{Definition}\label{def:aldual}
%   Let $I=(\mathbf{x}^{\alpha_1},\ldots, \mathbf{x}^{\alpha_q})\subset K[\mathbf{x}]= K[x_1,\ldots,x_n]$ be a square-free
%   monomial ideal, with $\alpha_i=(\alpha_{i_1},\ldots,\alpha_{i_n})\in \{0,1\}^n$. The \emph{Alexander dual}
%   of $I$ is the ideal
%   
%   $$I^*=\mathop\bigcap\limits_{i=1}^q \mathfrak{m}_{\alpha_i},$$
%   
%   where $\mathfrak{m}_{\alpha_i}=(x_j:\alpha_{i_j}=1)$.
% \end{Definition}

\begin{Proposition}\label{prop:dual}
Let $S=K[x_1,\ldots, x_n,y_1,\ldots,y_m]$, then 
\begin{equation}
(\sum_{i=1}^s  I_{q_i} J_{r_i})^*=I_{n-q_1+1} + \sum_{i=1}^{s-1}  I_{n-q_{i+1}+1} J_{m-r_i+1} + J_{m-r_s+1}. 
\end{equation}
\end{Proposition}

\begin{proof}
We prove the assertion by induction on $s$. If $s=1$ we have that either $q_1=0$ (resp. $r_1=0$) and $r_1\neq 0$ (resp. $q_1\neq 0$) or $q_1\neq 0$ and $r_1\neq 0$. The assertion for the first and the second case follows respectively by Proposition 2.2  and Corollary 2.4 of \cite{Ri}. Now suppose that 
\[
(\sum_{i=1}^{s-1}  I_{q_i} J_{r_i})^*=I_{n-q_1+1} + (\sum_{i=1}^{s-2}  I_{n-q_{i+1}} J_{m-r_i+1}) + J_{m-r_{s-1}+1}. 
\]
By equation \ref{eq:sumdual} we have
\[
 (I_{q_s}J_{r_s}+\sum_{i=1}^{s-1}  I_{q_i} J_{r_i})^*=(I_{q_s}J_{r_s})^*\cap(\sum_{i=1}^{s-1}  I_{q_i} J_{r_i})^*
\]
that is equal to, by Corollary 2.4 of \cite{Ri} and induction hypothesis, 
\begin{equation}\label{int:1}
 (I_{n-q_s+1}+J_{m-r_s+1})\cap(I_{n-q_1+1} + \sum_{i=1}^{s-2}  I_{n-q_{i+1}+1} J_{m-r_i+1} + J_{m-r_{s-1}+1}).
\end{equation}
We observe,  since $s>1$ and $q_s>q_i$ for all $i$, that $q_s\neq 0$. Let $H=I_{n-q_1+1} + \sum_{i=1}^{s-2}  I_{n-q_{i+1}+1} J_{m-r_i+1}$.
If we apply the modular law to \ref{int:1} we have
\[
 (I_{n-q_s+1} + J_{m-r_s+1})\cap H + (I_{n-q_s+1} + J_{m-r_s+1})\cap J_{m-r_{s-1}+1}.
\]
Since by hypothesis $q_i<q_s\leq n$ we  have $I_{n-q_s+1} \supset H$ and observing that $J_{m-r_{s}+1}\supset J_{m-r_{s-1}+1}$, the assertion follows easily.
\end{proof}
\begin{Remark}
 By Proposition \ref{prop:dual} we have that the class of mixed pruduct ideals with a finite set of summand is closed under Alexander duality (see also \cite{Ri}, Remark 2.5).
\end{Remark}

% \begin{Corollary}\label{cor:prdec}
% Let $X_i\subset 2^{\{x_1,\ldots,x_n\}}$ such that $T\in X_i$ if $|T|=i$ and let $Y_j\subset 2^{\{y_1,\ldots,y_n\}}$ such that $U\in Y_j$ if $|U|=j$. Then
% \[
% \begin{array}{lll}
% \sum_{i=1}^s  I_{q_i} J_{r_i} & = & \bigcap_{S_x\in X_{n-q_1+1}} (S_x)\\
%  & \bigcap_{i=1}^{s-1} & \bigcap_{S_x\in X_{n-q_{i+1}+1}, S_y\in Y_{m-r_{i}+1}} (S_x\cup S_y)\\
%  &    & \bigcap_{S_y\in Y_{m-r_s+1}} (S_y)
% \end{array}
% \]
% \end{Corollary}
\begin{Corollary}\label{cor:prdec}
Let $S=K[x_1,\ldots, x_n,y_1,\ldots,y_m]$ and let
\[
 \mathcal{X}_i=\{X\subset \{x_1,\ldots,x_n\}: |X|=i\},  \mathcal{Y}_j=\{Y\subset \{y_1,\ldots,y_m\}: |Y|=j\},
\]
with $\mathcal{X}_i=\emptyset$ if $i>n$ and $\mathcal{Y}_j=\emptyset$ if $j>m$. 
Then
\[
\sum_{i=1}^s  I_{q_i} J_{r_i}  = \PP_x \cap \PP_{xy}  \cap \PP_y
\]
where
\[
 \PP_{x}=\mathop\bigcap_{X\in \mathcal{X}_{n-q_1+1} } (X),\hspace{1cm} \PP_{y}=\mathop\bigcap_{Y\in \mathcal{Y}_{m-r_s+1} } (Y),
\]
\[
 \PP_{xy}=\bigcap_{i=1}^{s-1} \left(\bigcap_{X,Y} ((X)+(Y))\right),\, X\in \mathcal{X}_{n-q_{i+1}+1},\,Y\in \mathcal{Y}_{m-r_{i}+1} .
\]

\end{Corollary}

\begin{proof}
 By Alexander duality and Proposition \ref{prop:dual} we have that
\[
\sum_{i=1}^s  I_{q_i} J_{r_i}=(I_{n-q_1+1} + \sum_{i=1}^{s-1}  I_{n-q_{i+1}+1} J_{m-r_i+1} + J_{m-r_s+1})^*.
\]
By equation \ref{eq:alexdual} the assertion follows.
\end{proof}

\begin{Corollary}\label{cor:unmixed}
Let $S=K[x_1,\ldots,x_n,y_1,\ldots,y_m]$ and let  $\sum_{i=1}^s  I_{q_i} J_{r_i}$  
be the mixed product ideal on the ring $S$.
Let $h=\height \sum_{i=1}^s  I_{q_i} J_{r_i}$, then the ideal is unmixed if and only if the following conditions are satisfied:
\begin{enumerate}
 \item $m+n-(q_{i+1}+r_i)+2=h$, $\forall i=1,\ldots,s-1$;
 \item if $q_1> 0$ then $n-q_1+1=h$;
 \item if $r_s> 0$ then $m-r_s+1=h$.
\end{enumerate}
\end{Corollary}

\begin{Definition}\label{def:qrvec}
Let $S=K[x_1,\ldots,x_n,y_1,\ldots,y_m]$ and let  $\sum_{i=1}^s  I_{q_i} J_{r_i}$  
%with $0\leq q_1< \cdots < q_s\leq n$ and  $0\leq r_s< \cdots < r_s\leq m$   
be the mixed product ideal on the ring $S$. We define $s'\in \NN$ such that
\[
 s'=\left\{\begin{array}{ll}
            s+1 & \mbox{ if } q_1> 0 \mbox{ and } r_s> 0\\
            s & \mbox{ if } q_1> 0 \mbox{ and } r_s= 0 \mbox{ or }q_1= 0\mbox{ and } r_s> 0 \\
            s-1 & \mbox{ if } q_1=r_s=0
           \end{array}
    \right.
\]
and the two vectors $\bar{q}=(q(1),\ldots,q(s'))$, $\bar{r}=(r(1),\ldots,r(s'))\in \ZZ_{\geq 0}^{s'}$ such that
\[
 q(i)=\left\{\begin{array}{ll}
            q_i -1 & \mbox{ if } q_1> 0\\
            q_{i+1}-1 & \mbox{ if } q_1= 0
           \end{array}
    \right.
\]
\[
 r(i)=\left\{\begin{array}{ll}
            r_{i-1} -1 & \mbox{ if } q_1> 0\\
            r_i-1 & \mbox{ if } q_1= 0
           \end{array}
    \right.
\]
with $i=1,\ldots,s'$ and $r_0=m+1$ and $q_{s+1}=n+1$.
\end{Definition}

\begin{Proposition}\label{rem:partition}
Let $S=K[x_1,\ldots, x_n,y_1,\ldots,y_m]$ and let $I_\Delta=\sum_{i=1}^s  I_{q_i} J_{r_i}$ be a mixed product ideal.
Using the notation of Definition \ref{def:qrvec} there exists a partition of $\FF(\Delta)$, $\FF(\Delta)=\FF_1(\Delta)\cup \ldots \cup \FF_{s'}(\Delta)$ such that

\begin{equation}\label{eq:partition}
\begin{split}
\FF_k(\Delta) = &  \{\{x_{i_1},\ldots, x_{i_{q(k)}},y_{j_1},\ldots, y_{j_{r(k)}}\}: \\
               &      1\leq i_1<\ldots<i_{q(k)}\leq n,\, 1\leq j_1<\ldots<j_{r(k)}\leq m\},
\end{split}
\end{equation}
with $k=1,\ldots,s'$. 
\end{Proposition}
\begin{proof}
 By Corollary \ref{cor:prdec} and equation \ref{eq:stanleyideal} the assertion follows.

\end{proof}
From now on we associate to a mixed product ideal $\sum_{i=1}^s  I_{q_i} J_{r_i}$ the value $s'\in \NN$ and the vectors $\bar{q}=(q(1),\ldots,q(s'))$, $\bar{r}=(r(1),\ldots,r(s'))\in \ZZ_{\geq 0}^{s'}$ defined in \ref{def:qrvec}. We also give, for the sake of completeness,  a way to compute the sequences $0\leq q_1<\ldots<q_s\leq n$, $0\leq r_s<\ldots<r_1\leq m$ by the vectors $\bar{q}=(q(1),\ldots,q(s'))$ and $\bar{r}=(r(1),\ldots,r(s'))$. 

\begin{Definition}\label{def:qrseq}
Let $s'\in \NN$, $\bar{q}=(q(1),\ldots,q(s'))$, $\bar{r}=(r(1),\ldots,r(s'))\in\ZZ_{\geq 0}^{s'}$, with $0\leq q(1) <\ldots<q(s')\leq n$, $0\leq r(s') <\ldots<r(1)\leq m$. We define $s\in \NN$ such that
\[
 s=\left\{\begin{array}{ll}
            s'-1 & \mbox{ if } r(1)=m \mbox{ and } q(s')=n\\
            s' & \mbox{ if } r(1)=m \mbox{ and } q(s')< n \mbox{ or }r(1)<m \mbox{ and } q(s')=n \\
            s'+1 & \mbox{ if } r(1)<m \mbox{ and } q(s')< n
           \end{array}
    \right.
\]
and the two sequences $0\leq q_1<\ldots<q_s\leq n$, $0\leq r_s<\ldots<r_1\leq m$ such that
\[
 q_i=\left\{\begin{array}{ll}
            q(i)+1 & \mbox{ if } r(1)=m\\
            q(i-1)+1 & \mbox{ if } r(1)<m
           \end{array}
    \right.
\]
\[
 r_i=\left\{\begin{array}{ll}
            r(i+1) +1 & \mbox{ if } r(1)=m\\
            r(i)+1 & \mbox{ if } r(1)<m
           \end{array}
    \right.
\]
with $i=1,\ldots,s$ and $q(0)=r(s'+1)=-1$.
\end{Definition}

\begin{Lemma}\label{lem:intersection}
Let $S=K[x_1,\ldots, x_n,y_1,\ldots,y_m]$, $I_\Delta=\sum_{i=1}^s  I_{q_i} J_{r_i}$ be a mixed product ideal and keep the notation of Proposition \ref{rem:partition}. Then for each $F\in \FF_i(\Delta)$ and for each $G\in \FF_j(\Delta)$ with $1\leq i<j\leq s'$  we have
\[\dim F\cap G \leq q(i)+r(j)-1.\]
\end{Lemma}
\begin{proof}
 By Proposition \ref{rem:partition} we have $|F|=q(i)+r(i)$ and $|G|=q(j)+r(j)$. By the ordering in \ref{order:q} and \ref{order:r} and the Definition \ref{def:qrvec} we have $q(i)<q(j)$, $r(i)>r(j)$ for all $1\leq i<j\leq s'$  and the assertion follows.
\end{proof}

\begin{Lemma}\label{lem:strongpath}
Let $S=K[x_1,\ldots, x_n,y_1,\ldots,y_m]$ and let $I_\Delta=\sum_{i=1}^s  I_{q_i} J_{r_i}$ be a mixed product ideal and keep the notation of Proposition \ref{rem:partition}.
Let $I_\Delta$ be unmixed and let $F\in\FF_{i}(\Delta)$ and $G\in\FF_{j}(\Delta)$  with $\dim F \cap G=\dim \Delta-1$. If  $i<j$  (resp. $i>j$ ) then
\begin{enumerate}
 \item $j=i+1$ (resp. $j=i-1$);
 \item $q(i+1)=q(i)+1$ (resp. $q(i-1)=q(i)-1$);
 \item $r(i+1)=r(i)-1$ (resp. $r(i-1)=r(i)+1$).
\end{enumerate}
\end{Lemma}
\begin{proof}
 (1) We assume $i<j$. By Lemma \ref{lem:intersection} and since $\Delta$ is pure we have the following inequality
\begin{equation}\label{eq:ineq}
 \dim F\cap G =\dim \Delta-1= q(i)+r(i)-2 \leq  q(i)+r(j)-1.
\end{equation}

Since $r(j)>r(i)$ and by the inequality \ref{eq:ineq} we obtain $r(j)<r(i)\leq r(j)+1$ that is $r(i)=r(j)+1$. Therefore $j=i+1$ and $r(i+1)=r(i)-1$. By similar arguments we easily complete the proof of the assertion.
\end{proof}

\begin{Lemma}\label{lem:shellable}
Let $S=K[x_1,\ldots, x_n,y_1,\ldots,y_m]$ and let $I_\Delta=\sum_{i=1}^s  I_{q_i} J_{r_i}$ be a mixed product ideal.
If $q(i+1)=q(i)+1$ for $i=1,\ldots,s'-1$ then $\Delta$ shellable.
\end{Lemma}
\begin{proof}
We consider the partition $\FF(\Delta)=\FF_1(\Delta)\cup \ldots \cup \FF_{s'}(\Delta)$ defined in Proposition \ref{rem:partition}. We set a linear order $\prec$ on the facets $\FF(\Delta)$ such that $F\prec G$ with $F\in\FF_k(\Delta)$, $G\in\FF_{k'}(\Delta)$ if either  $k<k'$ or $k=k'$ with  
\[
 F=\{x_{i_1},\ldots, x_{i_{q(k)}}, y_{j_1},\ldots, y_{j_{r(k)}}\},\, G=\{x_{i'_1},\ldots, x_{i'_{q(k)}}, y_{j'_1},\ldots, y_{j'_{r(k)}}\},
\]
$1\leq i_1< \ldots <i_{q(k)}\leq n,\, 1\leq j_1<\ldots<j_{r(k)}\leq m$, $1\leq i'_1< \ldots <i'_{q(k)}\leq n,\, 1\leq j'_1<\ldots<j'_{r(k)}\leq m$
and there exists $p$, $1\leq p\leq q(k)$, such that  $i_k=i'_k$ for $k=1,\ldots,p-1$  but $i_p<i_p'$ or $i_k=i'_k$ for all $k=1,\ldots,q(k)$ and exists $p'$, $1\leq p'\leq r(k)$, such that  $j_k=j'_k$ for $k=1,\ldots,p'-1$  but $j_{p'}<j'_{p'}$.

Suppose $F\prec G$ with $F\in \FF_{i}(\Delta)$ and $G\in \FF_{j}(\Delta)$ with $i<j$. Since $q(i)<q(j)$ there exists  $x_k\in G\setminus F$. Now let $G_k=G\setminus \{x_k\}$. We observe that there exists  $F_k\in \FF_{j-1}(\Delta)$ such that $F_k\supset G_k$, in fact  by hypothesis $q(j-1)=q(j)-1$ and $r(j-1)>r(j)$. Hence $G\setminus F_k=\{x_k\}$.

Suppose $F\prec G$ with $F, G\in \FF_{i}(\Delta)$. We may assume $x_k\in G\setminus F$, in fact if such $x_k$ does not exist we can consider the case $y_k\in G\setminus F$ in an analogous way. Since $F\prec G$ there exists $x_{k'}\in F\setminus G$ such that $k'<k$.  We set $F_k=(G\setminus\{x_k\}) \cup \{x_{k'}\}$. We observe that $F_k\in \FF_{i}(\Delta)$, $F_k\prec G$ and $G\setminus F_k=\{x_k\}$. The assertion follows.
\end{proof}

By the same argument we have the following
\begin{Lemma}\label{lem:shellable1}
Let $S=K[x_1,\ldots, x_n,y_1,\ldots,y_m]$ and let $I_\Delta=\sum_{i=1}^s  I_{q_i} J_{r_i}$ be a mixed product ideal.
If $r(i+1)=r(i)-1$ for $i=1,\ldots,s'-1$ then $\Delta$ is shellable.
\end{Lemma}

\begin{Theorem}\label{theo:strconn}
Let $S=K[x_1,\ldots, x_n,y_1,\ldots,y_m]$, $I_\Delta=\sum_{i=1}^s  I_{q_i} J_{r_i}$ be a mixed product ideal, $K[\Delta]=S/I_\Delta$.
The following conditions are equivalent:
\begin{enumerate}
% \item $\Delta$ satisfies the {\em increasing by $1$ property for $x$ and  $y$};
 \item $q(i+1)=q(i)+1$ and $r(i+1)=r(i)-1$ for $i=1,\ldots,s'-1$;
 \item $\Delta$ is pure shellable;
 \item $K[\Delta]$ is Cohen-Macaulay;
 \item $\Delta$ is strongly connected.
\end{enumerate}
\end{Theorem}
\begin{proof}
(1)$\Rightarrow$(2). By Lemma \ref{lem:shellable} (or equivalently \ref{lem:shellable1}) we have that $K[\Delta]$ is shellable.
We observe that $q(i+1)+r(i+1)=q(i)+r(i)$ for $i=1,\ldots,s'-1$. Hence  $\Delta$ is pure.

(2)$\Rightarrow$(3). Always true.

(3)$\Rightarrow$(4). Always true.

(4)$\Rightarrow$(1). Let $i=1,\ldots,s'-1$ and let $F\in \FF_i(\Delta)$ and $G\in \FF_{i+1}(\Delta)$. Since $\Delta$ is strongly connected there exists a sequence of facets $F=F_0,F_1,\ldots,F_t=G$ such that $\dim F_{k}\cap F_{k+1}=\dim \Delta-1$ for $k=0,\ldots,t-1$. We observe that there exists $k\in\{0,\ldots,t-1\}$ such that
\[
 F_k\in \bigcup_{j\leq i} \FF_j(\Delta),\hspace{1cm} F_{k+1}\in \bigcup_{j\geq i+1} \FF_j(\Delta).
\]
Let $F_k\in \FF_{i-d}(\Delta)$ and $F_{k+1}\in \FF_{i+1+d'}(\Delta)$ with $0\leq d\leq i-1$, $0\leq d'\leq s'-i-1$.  Since $q(i-d)\leq q(i)-d$ and $r(i+1+d')\leq r(i)-1-d'$, by Lemma \ref{lem:intersection} we obtain
\[
 \dim F_k\cap F_{k+1}\leq q(i)+r(i)-(d+d')-2.
\]
On the other hand $\dim F_k\cap F_{k+1}=\dim \Delta -1=q(i)+r(i)-2$. Hence $d=d'=0$. The assertion follows by Lemma \ref{lem:strongpath}.
\end{proof}

\section{Sequentially Cohen-Macaulay mixed product ideals}\label{sec:SCM}
In this section we classify the sequentially Cohen-Macaulay mixed product ideals. We recall some definitions and results useful for our purpose and we continue to use the notation defined in section \ref{sec:Alex}.

\begin{Definition}
Let $K$ be a field, $S=K[x_1,\ldots,x_n]$ be a polynomial ring. A graded $S$-module is called {\em sequentially Cohen-Macaulay} (over $K$), if there exists a finite filtration of graded $S$-modules
\[
 0=M_0\subset M_1\subset \cdots\subset M_t=M
\]
such that each $M_i/M_{i-1}$ is Cohen-Macaulay, and the Krull dimensions of the quotients are increasing:
\[
 \dim(M_1/M_0)<\dim(M_2/M_1)< \cdots<\dim(M_{t}/M_{t-1}).
\]
\end{Definition}

\begin{Definition}
 Let $\Delta$ be a simplicial complex then we define the pure simplicial complexes $\Delta^{[l-1]}$  whose facets are
\[
 \FF(\Delta^{[l-1]})=\{F\in\Delta:\dim(F)=l-1\},\hspace{1cm} 0\leq l\leq \dim(\Delta)+1. 
\]
\end{Definition}

A fundamental result about sequentially Cohen-Macaulay Stanley-Reisner rings $K[\Delta]$ is the following

\begin{Theorem}[\cite{Du}]\label{th:duval}
 $K[\Delta]$ is sequentially Cohen-Macaulay if and only if $K[\Delta^{[l-1]}]$ is Cohen-Macaulay for  $0\leq l\leq \dim(\Delta)+1$.
\end{Theorem}

\begin{Remark}\label{rem:subl}
Let $S=K[x_1,\ldots, x_n,y_1,\ldots,y_m]$, $I_\Delta=\sum_{i=1}^s  I_{q_i} J_{r_i}$ be a mixed product ideal, $K[\Delta]=S/I_\Delta$ and let $\FF(\Delta)$ be partitioned as shown in Proposition \ref{rem:partition}, that is $\FF(\Delta)=\FF_1(\Delta)\cup \ldots \cup \FF_{s'}(\Delta)$ such that
\begin{equation*}
\begin{split}
\FF_k(\Delta) = &  \{\{x_{i_1},\ldots, x_{i_{q(k)}},y_{j_1},\ldots, y_{j_{r(k)}}\}: \\
                &      1\leq i_1<\ldots<i_{q(k)}\leq n,\, 1\leq j_1<\ldots<j_{r(k)}\leq m\},
\end{split}
\end{equation*}
with $k=1,\ldots,s'$. If we set an $l$ with $0\leq l\leq \dim(\Delta)+1$ then for each $k\in\{1,\ldots, s'\}$ we have that $\FF_k(\Delta^{[l-1]})=\FF_{k1} \cup \ldots \cup \FF_{kt_{k}}$ where 
\begin{equation*}
\begin{split}
\FF_{kj} &= \{\{x_{i_1},\ldots, x_{i_{q_k(j)}},y_{j_1},\ldots, y_{j_{r_k(j)}}\}: \\
               &     1\leq i_1<\ldots<i_{q_k(j)}\leq n,\, 1\leq j_1<\ldots<j_{r_k(j)}\leq m\}  \mbox{ with } j=1,\ldots,t_k,
\end{split}
\end{equation*}
satisfies the following properties:
\begin{enumerate}
 \item $q_k(t_k)=\min\{q(k),l\}$,
 \item $r_k(1)=\min\{r(k),l\}$,
 \item $q_k(i)=q_k(i+1)-1$, $r_k(i+1)=r_k(i)-1$ for $i=1,\ldots,t_k-1$. 
\end{enumerate}
\end{Remark}

\begin{Definition}
  Let $\bar{q}=(q(1),\ldots,q(s'))$, $\bar{r}=(r(1),\ldots,r(s'))\in \ZZ^{s'}_{\geq 0}$, we define the following function $\sigma:\{1,\ldots,s'\}\rightarrow \ZZ_{\geq 0}$  
\[
\sigma(i)=q(i)+r(i). 
\]
\end{Definition}

\begin{Lemma}\label{lem:scm}
Let $S=K[x_1,\ldots, x_n,y_1,\ldots,y_m]$, $I_\Delta=\sum_{i=1}^s  I_{q_i} J_{r_i}$ be a mixed product ideal  and let $K[\Delta]=S/I_\Delta$.
If $K[\Delta]$ is sequentially Cohen-Macaulay then
\begin{enumerate}
 \item for all $i\in\{1,\ldots,s'-1\}$ either $q(i)=q(i+1)-1$ or $r(i)=r(i+1)+1$;
 \item there exists $k\in \{1,\ldots,s'\}$ such that $\sigma(1)\leq \sigma(2)\leq \ldots\leq \sigma(k)\geq \sigma(k+1)\geq \sigma(s')$.
\end{enumerate}
\end{Lemma}
\begin{proof}
 If $K[\Delta]$ is sequentially Cohen-Macaulay then $K[\Delta^{[l-1]}]$ is Cohen-Macaulay for $0\leq l\leq \dim \Delta+1$. Hence $\Delta^{[l-1]}$ is strongly connected for $0\leq l\leq \dim \Delta+1$. We observe that if we negate property (1) (resp. (2)) we find an $l$ such that $\Delta^{[l-1]}$ is not strongly connected.

 (1) We suppose that there exists $k$ with $1\leq k\leq s'-1$ such that 
\[
 q(k)<q(k+1)-1\mbox{ and } r(k)>r(k+1)+1.
\]
Let $l=\min\{\sigma(k), \sigma(k+1)\}$ and we assume that $l=\sigma(k)$. We observe that
\begin{equation}\label{eq:subpar}
 \FF(\Delta^{[l-1]})=\bigcup_{i=1}^{s'} \FF_i (\Delta^{[l-1]})
\end{equation}
where the union is not disjoint. Since $l=\sigma(k)\leq \sigma(k+1)$ we have $\FF_k (\Delta^{[l-1]})=\FF_k (\Delta)$ and $\FF_{k+1} (\Delta^{[l-1]})\neq \emptyset$ .

We show that  for all $F\in\FF_i(\Delta)$ with $\FF_i(\Delta^{[l-1]})\neq \emptyset$ and for all $G\in\FF_j(\Delta)$ with  $\FF_j(\Delta^{[l-1]})\neq \emptyset$ with $1\leq i\leq k <j\leq s'$ we have that
\begin{equation}\label{diseq:dim}
\dim F\cap G < \dim \Delta^{[l-1]}-1.  
\end{equation}
If the inequality \ref{diseq:dim} is satisfied also the facets $F'\subset F$ with $F'\in \FF_i(\Delta^{[l-1]})$ and  $G'\subset G$ with $G'\in \FF_j(\Delta^{[l-1]})$ inherit this property. 

Hence $\Delta^{[l-1]}$ is not strongly connected.

By Lemma \ref{lem:intersection}  we have that $\dim F\cap G\leq q(i)+r(j)-1 \leq q(k)+r(j)-1$. Since $r(k)>r(k+1)+1$ we have $q(k)+r(k)>q(k)+r(k+1)+1$ that is 
\[
l-2= q(k)+r(k)-2>q(k)+r(j)-1\geq \dim F\cap G.
\]
The case $l=\sigma(k+1)$ follows by similar arguments.

(2) We suppose that there exist $k$, $k^-$, $k^+$, with $1\leq k^-<k<k^+\leq s'$ such that 
\[
 \sigma(k^-) > \sigma(k)< \sigma (k^+).
\]
Let $l=\min\{\sigma (k^-),\sigma(k^+)\}$ and we assume that $l=\sigma(k^-)$. 

Hence $\FF_{k^-} (\Delta^{[l-1]})=\FF_{k^-} (\Delta)$ and, since $l\leq \sigma(k^+)$, $\FF_{k^+} (\Delta^{[l-1]})\neq \emptyset$. By Lemma \ref{lem:intersection} and using similar arguments of (1) it is easy to show that, for all $F\in\FF_i(\Delta)$ with $\FF_i(\Delta^{[l-1]})\neq \emptyset$ and  for all $G\in\FF_j(\Delta)$ with  $\FF_j(\Delta^{[l-1]})\neq \emptyset$ with $1\leq i< k <j\leq s'$,
\[
\dim F\cap G \leq q(i)+r(j)-1< q(k)+r(k)-1< l-1.  
\]
The assertion follows since $\FF_{k} (\Delta^{[l-1]})=\emptyset$.
\end{proof}

We come to the main result of this section.
\begin{Theorem}\label{theo:SCMqr}
Let $S=K[x_1,\ldots, x_n,y_1,\ldots,y_m]$, $I_\Delta=\sum_{i=1}^s  I_{q_i} J_{r_i}$ be a mixed product ideal  and let $K[\Delta]=S/I_\Delta$.
The following conditions are equivalent:
\begin{enumerate}

 \item $K[\Delta]$ is sequentially Cohen-Macaulay.
 \item The following conditions hold:
\begin{enumerate}
 \item $q(i)=q(i+1)-1$ or $r(i)=r(i+1)+1$ with $i=1,\ldots,s'-1$;
 \item there exists $k\in \{1,\ldots,s'\}$ such that $\sigma(1)\leq \sigma(2)\leq \ldots\leq \sigma(k)\geq \sigma(k+1)\geq \sigma(s')$.
\end{enumerate}
 \end{enumerate}
\end{Theorem}
\begin{proof}
(1)$\Rightarrow$(2).  See Lemma \ref{lem:scm}.

(2)$\Rightarrow$(1). We need to show that for all $l$ with $0\leq l\leq \dim \Delta+1$ we have $K[\Delta^{[l-1]}]$ is Cohen-Macaulay.
 %We observe that the facets of the subcomplex $\Delta^{[l-1]}$ of $\Delta$ are the subsets of cardinality $l$ of the facets of $\FF(\Delta)=\bigcup_{i=1}^{s'} \FF_i (\Delta)$ such that $\sigma(i)\geq l$. 
Let $\Delta'=\Delta^{[l-1]}$, by Remark \ref{rem:subl} we have that 
\begin{equation}\label{eq:par}
 \FF(\Delta')=\bigcup \FF_{kj}\mbox{ with }k=1,\ldots,s',j=1,\ldots,t_k,
\end{equation}
where
\[
 \FF_{kj}\cap \FF_{k'j'}=\left\{\begin{array}{cl}
	    \FF_{kj}=\FF_{k'j'} & \mbox{ if } q_k(j)=q_{k'}(j')\\
	    \emptyset           & \mbox{ if } q_k(j)\neq q_{k'}(j')
            \end{array}
    \right.
\]
for all $k,k'\in \{1,\ldots,s'\}$, $j=1,\ldots,t_k$, $j'=1,\ldots,t_{k'}$.
If we remove the redundant elements in \ref{eq:par} and sort the remaining ones in an increasing order by $q_k(j)$ with $k=1,\ldots,s'$ and $j=1,\ldots,t_k$, we obtain a partition, %in analogous way to that one Proposition \ref{rem:partition} a partition 
with $\bar{q}'=(q'(1),\ldots,q'(t'))$, $\bar{r}'=(r'(1),\ldots,r'(t'))$ and $q'(i)<q'(i+1)$ for $i=1,\ldots,t'-1$.
Since $\Delta'$ is pure by definition, it is sufficient to show that $q'(i+1)=q'(i)+1$ for $i=1,\ldots,t'-1$ by Theorem \ref{theo:strconn}.

Let $q'(i)$ be an entry of the vector $\bar{q}'$ with $i=1,\ldots,t'-1$, then $q'(i)<l$ and there exists $q_k(j)$ related to \ref{eq:par} with $q'(i)=q_k(j)$ with $k=1,\ldots,s'$ and $j=1,\ldots,t_k$.

If $j<t_k$ by property (3) of Remark \ref{rem:subl} we are done. If $j=t_k$ this implies that in the partition induced by \ref{eq:par} there exists $k'>k$ such that $\FF_{k'}(\Delta')\neq \emptyset$. Hence by the condition (1.b), $\sigma(k+1)\geq \min\{\sigma(k),\sigma(k')\}\geq l$, therefore $\FF_{k+1}(\Delta')\neq \emptyset$. By condition (1.a), if $q(k+1)=q(k)+1$  and by property (1) of Remark \ref{rem:subl} we have $q_{k+1}(t_{k+1})=q(k)+1$. If $q(k+1)\neq q(k)+1$ then $r(k+1)=r(k)-1$ and this implies by property (2) of Remark \ref{rem:subl} that $r_{k+1}(1)=r(k)-1$, hence $q_{k+1}(1)=l-(r(k)-1)=q(k)+1$.

\end{proof}


\begin{thebibliography}{00}
\bibitem{BH}{W. Bruns, J. Herzog}, 
 \textit{Cohen-Macaulay rings},
\newblock Cambridge Univ. Press, Cambridge, 1997.



\bibitem{Du} A. M. Duval, 
\newblock  \textit{Algebraic shifting and sequentially Cohen-Macaulay simplicial complexes}
\newblock Electron. J.Combin. 3 (1996), n.1, Research Paper 21.


\bibitem{Fa} Faridi, S.,
\newblock  \textit{Simplicial trees are sequentially Cohen--Macaulay}
J. Pure Appl. Algebra. \textbf{190} (2004), 121--136. 

\bibitem{HH} J. Herzog, T. Hibi,  
\newblock  \textit{Componentwise linear ideals}, Nagoya Math. J., \textbf{153} (1999), 141--153. 

\bibitem{HN} L. T. Hoa,  N. D. Tam,
\newblock  \textit{On some invariants of a mixed product of ideals},
\newblock  Arch. Math. \textbf{94} (2010), 327--337.

\bibitem{IR} C. Ionescu, G. Rinaldo,
\newblock  \textit{Some algebraic invariants related to mixed product ideals},
\newblock  Arch. Math. \textbf{91} (2008), 20--30.


\bibitem{MS}{E. Miller, B. Sturmfels}, 
\newblock \textit{Combinatorial Commutative Algebra}, Springer-Verlag, Berlin, 2005.
\newblock 


\bibitem{RV}{G. Restuccia, R. Villarreal}, 
\newblock \textit{On the normality of monomial ideals of mixed products},
\newblock Commun.  Algebra, \textbf{29} (2001), 3571--3580.

\bibitem{Ri} G. Rinaldo,
\newblock  \textit{Betti numbers of mixed product ideals},
\newblock  Arch. Math. \textbf{91} (2008), 416--426.
 


\bibitem{St} R.~P.~Stanley, 
 \textit{Combinatorics and Commutative Algebra, Second Edition}, 
 Birkh\"auser, Boston/Basel/Stuttgart, 1996. 



\bibitem{VV}{A. Van Tuyl, R. Villarreal}, 
\newblock \textit{Shellable graphs and sequentially Cohen--Macaulay bipartite graphs},
Journ. of Combinat. Th. Ser. A, \textbf{115}, (2008), 799--814. 

\bibitem{Vi}{R. Villarreal}, 
\newblock \textit{Monomial algebras},
\newblock Marcel Dekker, New-York, 2001.




\end{thebibliography}
\end{document}